\documentclass[10pt]{amsart}
\usepackage{amsfonts}
\usepackage{amssymb}
\usepackage{amsmath}
\usepackage[arrow,matrix,curve]{xy}

\newcommand{\C}{{\mathbb C}}

\newcommand{\Z}{{\mathbb Z}}
\newcommand{\N}{{\mathbb N}}

\newcommand{\E}{{\mathbb E}}

\newcommand{\F}{{\mathbb F}}

\newcommand{\D}{{\mathbb D}}

\newtheorem{theorem}{Theorem}[section]
\newtheorem{lemma}[theorem]{Lemma}
\newtheorem{remark}[theorem]{Remark}

\newtheorem{corollary}[theorem]{Corollary}
\newtheorem{proposition}[theorem]{Proposition}
\newtheorem{definition}[theorem]{Definition}

\newtheorem{modification}[theorem]{Modification}

\def\cal{\mathcal}

\newcommand{\call}[0]{{\cal L}}
\newcommand{\calg}[0]{{\cal G}}

\newcommand{\cala}[0]{{\cal A}}
\newcommand{\cale}[0]{{\cal E}}
\newcommand{\calk}[0]{{\cal K}}

\newcommand{\calm}[0]{{\cal M}}

\begin{document}

\title[Equivariant $KK$-theory]{The universal property of inverse semigroup equivariant $KK$-theory}
\author[B. Burgstaller]{Bernhard Burgstaller}
\address{Departamento de Matematica, Universidade Federal de Santa Catarina, CEP 88.040-900 Florian\'opolis-SC, Brasil}
\email{bernhardburgstaller@yahoo.de}

\begin{abstract}
Higson proved that every homotopy invariant, stable and
split exact functor
from the category of $C^*$-algebras to an additive
category factors through Kasparov's $KK$-theory.
By adapting a group equivariant generalization of this result 
by Thomsen,
we
generalize Higson's result to the inverse semigroup equivariant setting.

%
\end{abstract}

\maketitle

\section{Introduction}

In \cite{cuntz1984}, Cuntz noted that if $F$ is a homotopy invariant, stable
and split exact functor from the category 
of separable $C^*$-algebras
to the category of abelian groups then Kasparov's $KK$-theory acts on $F$, that is,
every element of $KK(A,B)$ induces a natural map $F(A) \rightarrow F(B)$.
Higson \cite{higson}, on the other hand,
developed Cuntz' findings further and proved that
every such functor $F$ factorizes through
the category ${\bf K}$ consisting of separable $C^*$-algebras as the object class
and $KK$-theory together with the Kasparov product 
as the
morphism class, that is, $F$ is the composition $\hat F \circ \kappa$ of a universal functor
$\kappa$ from the class of $C^*$-algebras
to ${\bf K}$ and a functor $\hat F$
from ${\bf K}$ to abelian groups.

In \cite{thomsen}, Thomsen generalized Higson's findings
to the group equivariant setting by replacing everywhere
in the above statement algebras by equivariant algebras, $*$-homomorphisms by equivariant $*$-homomorphisms
and $KK$-theory by equivariant $KK$-theory (the proof is however far from
such a straightforward replacement).
Meyer \cite{meyer2000} used a different approach in generalizing Higson's result,
and generalized it to the setting of action groupoids $G \ltimes X$.
%

In this note we extend Higson's universality result to the inverse
semigroup equivariant setting.
Contrary to the difficulties in generalizing Higson's proof
to the group equivariant setting, our generalization is
a simple adaption of Thomsen's proof.


Indeed, it mainly considers only single elements of the group and
the composition law in the group plays a subsidiary role, so
that the differences between inverse semigroups and groups
are moderate.

More specifically, the following will be shown.
Let $G$ be a countable unital inverse semigroup
and denote by ${\bf C^*}$ the category consisting of (ungraded) separable $G$-equivariant $C^*$-algebras
as objects and $G$-equivariant $*$-homomorphisms as morphisms.
Denote by {\bf Ab} the category of abelian groups. A functor $F$
from ${\bf C^*}$ to {\bf Ab} is called {\em stable} if $F(\varphi)$
is an isomorphism for every $G$-equivariant corner embedding
$\varphi:A \rightarrow A \otimes \calk$, where $A$ is an $G$-$C^*$-algebra,
$\calk$ denotes the
compacts on some separable Hilbert space, and $A \otimes \calk$
is a $G$-algebra where the $G$-action may be arbitrary and not necessarily
diagonal.
Similarly, $F$ is called {\em homotopy invariant} if any two homotopic morphisms
$\varphi_0,\varphi_1:A \rightarrow B$ in ${\bf C^*}$ satisfy $F(\varphi_0)= F(\varphi_1)$,
and
{\em split exact} if $F$ turns every short split exact sequence in
${\bf C^*}$ canonically into a short split exact sequence in {\bf Ab}.

Let ${\bf K^G}$ the denote the category consisting of separable $G$-equivariant $C^*$-algebras
as object class and the $G$-equivariant $KK$-theory group $KK^G(A,B)$
as the morphism set between two objects $A$ and $B$; thereby, composition of morphisms
is given by the Kasparov product ($g \circ f := f \otimes_B g$ for $f \in KK^G(A,B)$ and $g \in KK^G(B,C)$).
The one-element of the ring $KK^G(A,A)$ is denoted by $1_A$.

${\bf K^G}$ is an additive category ${\bf A}$, that means, the homomorphism
set ${\bf A}(A,B)$ is an abelian group and the composition law
is bilinear. We call a functor $F$ from ${\bf C^*}$ into an additive categtory ${\bf A}$
homotopy invariant, stable and split exact if the functor
$B \mapsto {\bf A}(F(A),F(B))$ enjoys these properties for every object $A$.
This notion is justified by the fact that this property is equivalent
to saying that $F$ itself satisfies these properties in 
the sense introduced above for ${\bf Ab}$,
see the properties (i)-(iii) in Higson \cite{higson}, page 269,
or Lemma \ref{lemmaFadditivecat}.

Let $\kappa$ 
denote the functor from ${\bf C^*}$ to ${\bf K^G}$
which is identical on objects and maps a morphism $g:A \rightarrow B$
to the morphism $g_*(1_A) \in KK^G(A,B)$.
Then we have the following proposition, theorems and corollary.

\begin{proposition}  \label{propositionFunctor}
Let $G$ be a countable unital inverse semigroup.
Let $A$ be an object in ${\bf C^*}$. Then $B \mapsto KK^G(A,B)$
is a homotopy invariant, stable and split exact covariant functor from
${\bf C^*}$ to {\bf Ab}.
\end{proposition}

\begin{theorem}   \label{theoremThomsen}
Let $G$ be a countable unital inverse semigroup.
Let $F$ be a homotopy invariant, stable and split exact covariant functor from
${\bf C^*}$ to {\bf Ab}. Let $A$ be an object in ${\bf C^*}$
and $d$ an element in $F(A)$.
Then there exists a unique natural transformation $\xi$
from the functor $B \mapsto KK^G(A,B)$ to the functor $F$
such that $\xi_A(1_A) = d$.
\end{theorem}

\begin{theorem}  \label{theoremHigson}
Let $G$ be a countable unital inverse semigroup.
Let $F$ be a homotopy invariant, stable and split exact covariant functor from
${\bf C^*}$ to an additive category {\bf A}.
Then there exists a unique functor $\hat F$ from ${\bf K^G}$ to {\bf A}
such that $F = \hat F \circ \kappa$, where $\kappa:{\bf C^*} \rightarrow {\bf K^G}$
denotes the canonical functor.
\end{theorem}

\begin{corollary}    \label{corollarynonunital}
The above results are also valid for countable non-unital inverse semigroups
$G$ and for countable discrete groupoids $G$.
\end{corollary}

%
%

A very brief overview of this note is as follows.
(We give further short summaries at the beginning of each section).

In Section \ref{sectionkktheory} we briefly recall the definitions of
inverse semigroup equivariant $KK$-theory.
In Section \ref{sectionProp} we prove Proposition \ref{propositionFunctor}.
In Section \ref{sectionPhi} we establish a Cuntz picture of $KK$-theory.
The core of
how to associate homomorphisms in the image of a
homotopy invariant, stable and split-exact
functor $F$ to
Kasparov elements
is explained in Section \ref{sectionPsi}.
Finally Theorem \ref{theoremThomsen}, Theorem \ref{theoremHigson}
and Corollary \ref{corollarynonunital}
are 
proved in Sections \ref{sectionxi}, \ref{sectionuniversal}
and \ref{sectionnonunital}, respectively.

Sections \ref{sectionProp}-\ref{sectionxi} present a slight adaption
of the content of Thomsen's paper \cite{thomsen}; this goes mostly without saying.
Sections \ref{sectionuniversal} is essentially taken from Higson's paper
\cite{higson}.


%
%
%

%


\section{Equivariant $KK$-theory}

\label{sectionkktheory}

Our reference for inverse semigroup equivariant $KK$-theory is \cite{burgiSemimultiKK}.
We shall however exclusively work with its slightly adapted variant called compatible $K$-theory, as in \cite{burgiKKrDiscrete},
but denote it by $KK^G$ rather than $\widehat{KK^G}$
as in \cite{burgiKKrDiscrete}. 
{All (exterior) tensor products are however ordinary as in \cite{burgiSemimultiKK}
and are not forced to be 
$C_0(X)$-balanced as in
\cite{burgiKKrDiscrete}!}
(The internal tensor products are automatically $C_0(X)$-balanced.)
For convenience of the reader we completely recall the basic definitions.

\begin{definition}
{\rm
Let $G$ denote a countable unital inverse semigroup.
The involution
in $G$ is denoted by $g \mapsto g^{-1}$ (determined by $g g^{-1} g = g$
and $g^{-1} g g^{-1}= g^{-1}$).
A semigroup homomorphism is said to be {\em unital} if it preserves the identity $1 \in G$.
To include also semigroups with a zero element, we insist that a 
semigroup homomorphism
preserves also the zero element $0 \in G$ if it exists.
}
\end{definition}

We denote the set of idempotent elements of $G$ by $E$.

%


\begin{definition}   \label{defCstar}
{\rm
A $G$-algebra $(A,\alpha)$ is a $\Z/2$-graded $C^*$-algebra $A$ with a
unital semigroup homomorphism
$\alpha: G \rightarrow \mbox{End}(A)$ such that
$\alpha_g$ respects the grading
and $\alpha_{g g^{-1}}(x) y = x \alpha_{g g^{-1}}(y)$
for all $x,y \in A$ and $g \in G$.
}
\end{definition}

Throughout we shall identify the $C^*$-algebras $\call_A(A)$
(adjoint-able operators on the Hilbert $A$-module $A$) and $\calm(A)$
(multiplier algebra of $A$) by a well-known $*$-isomorphism.

\begin{definition}   \label{defHilbert}
{\rm
A 
$G$-Hilbert $B$-module $\cale$ is a $\Z/2$-graded Hilbert
module over a $G$-algebra $(B,\beta)$ endowed with a unital
semigroup homomorphism $G \rightarrow \mbox{Lin}(\cale)$ (linear maps on $\cale$)
such that $U_g$ respects the grading and
\begin{itemize}
\item[(a)]
$\langle U_g(\xi),U_g(\eta)\rangle = \beta_g(\langle \xi,\eta \rangle)$

\item[(b)]
$U_g(\xi b) = U_g(\xi) \beta_g(b)$

\item[(c)]
$U_{g g^{-1}}(\xi) b = \xi \beta_{g g^{-1}}( b)$

\end{itemize}
for all $g \in G,\xi,\eta \in \cale$ and $b \in B$.
}
\end{definition}

\begin{lemma}     \label{lemmacenterU}
In the last definition, automatically $U_{g g^{-1}}$ is a self-adjoint projection in the center of the algebra
$\call(\cale)$.
\end{lemma}

\begin{proof}
For a positive approximate unit $(b_i) \subseteq B$ we compute
$$
\langle U_{g g^{-1}} \xi, \eta \rangle
\cong \langle \xi \beta_{g g^{-1}}(b), \eta \rangle
= \beta_{g g^{-1}}(b) \langle \xi, \eta \rangle 
\cong \beta_{g g^{-1}} \langle \xi, \eta \rangle
=  \langle \xi, U_{g g^{-1}} \eta \rangle,
$$
so that $U_{g g^{-1}}$ is seen to be self-adjoint.
This operator is in the center because
%
$$(U_{g g^{-1}} T (\xi)) b =  T (\xi) \beta_{g g^{-1}}(b)
=  T (\xi  \beta_{g g^{-1}}(b))
=  (T U_{g g^{-1}}(\xi)) b$$
for all $T \in \call(\cale), \xi \in \cale$ and $b \in B$.
\end{proof}

\begin{definition}    \label{defLaction}
{\rm
Given a $G$-Hilbert $B$-module $(\cale,U)$ we turn the $C^*$-algebra
$\call_B(\cale)$ to a $G$-algebra under the action $g(T):=U_g T U_{g^{-1}}$
for all $g \in G$ and $T \in \call(\cale)$.
}
\end{definition}


It is useful to notice that every $G$-algebra $(A,\alpha)$
is a $G$-Hilbert module over itself under the inner product $\langle a,b\rangle
= a^* b$; 
so we have all the identities of Definition \ref{defHilbert}
for $U: = \beta := \alpha$.
Actually Definitions \ref{defCstar} and \ref{defHilbert} are equivalent for $C^*$-algebras.
%
Hence we note that

\begin{lemma}   \label{lemmamultiae}
Let $(A,\alpha)$ be a $G$-algebra.
Every $\alpha_e \in \call_A(A) = \calm(A)$ for $e \in E$ is a self-adjoint projection
in the center of $\calm(A)$.
The application of $\alpha_e$ is given by multiplication, that is,
$\alpha_e(a) = a \alpha_e$ in $\calm(A)$.
\end{lemma}

\begin{definition}
{\rm
A $*$-homomorphism $\varphi:(A,\alpha) \rightarrow (B,\beta)$ between $G$-algebras 
is called {\em $G$-equivariant} if
$\varphi(\alpha_g(a))= \beta_g(\varphi(a))$ for all $a \in A, g \in G$.
}
\end{definition}

\begin{definition}
{\rm
A {\em $G$-Hilbert $A,B$-bimodule} over $G$-algebras $A$ and $B$
is a $G$-Hilbert $B$-module $\cale$ equipped with a $G$-equivariant $*$-homomorphism
$\pi:A \rightarrow \call(\cale)$ (the left module multiplication operator).
}
\end{definition}




\begin{definition}  \label{defCycle}
{\rm
Let $A$ and $B$ be $G$-algebras.
We define a $G$-equivariant Kasparov $A,B$-cycle
to be an ordinary Kasparov cycle
$(\cale,T)$ without $G$-action
(see \cite{kasparov1981,kasparov1988}) such that however
$\cale$ is a $G$-Hilbert $A,B$-bimodule and the operator $T \in \call(\cale)$ satisfies
\begin{equation}    \label{UT}
U_g T U_{g^{-1}} - T U_{g g^{-1}} \in \{S \in \call(\cale)|\, a S, S a \in \calk(\cale) \mbox{ for all } a \in A\}
\end{equation}
for all
$g \in G$. The Kasparov group $KK^G(A,B)$ is defined to be the collection 
of 
all $G$-equivariant Kasparov $A,B$-cycles divided by homotopy induced by 
$G$-equivariant Kasparov $A,B[0,1]$-cycles.
(Throughout, $B[0,1]:=B \otimes C([0,1])$.)
}
\end{definition}


We equip the multiplier algebra of a $G$-algebra
with a $G$-action as described in Definition \ref{defLaction}.
This is also the continuous extension of the $G$-action on $A$ to $\calm(A)$ in the strict topology. We
redundantly emphasize this again:

\begin{definition}   \label{defstrictlycont}
{\rm
Given a $G$-algebra $(A,\alpha)$, $\call_A(A)= \calm(A)$ becomes a $G$-algebra under the $G$-action $\overline \alpha: G \rightarrow \mbox{End}(\call_A(A))$
given by
$\overline{\alpha_g}(T) := \alpha_g \circ T \circ \alpha_{g^{-1}}$ for all $g \in G$
and $T \in \call_A(A)$.
}
\end{definition}

We also write $\overline \alpha: \calm(A) \rightarrow \calm(B)$ for the strictly
continuous extension of a $*$-homomomorphism $\alpha:A \rightarrow B$ of $C^*$-algebras.

\begin{definition}
{\rm
Write $\calk$ for the compact operators on a separable Hilbert space.
Call a $G$-algebra $(B,\beta)$ {\em stable} if there is a $G$-equivariant $*$-isomorphism
$(B,\beta) \rightarrow (B \otimes \calk, \beta \otimes \mbox{trivial})$.
}
\end{definition}


%
Since every $G$-algebra can be stabilized,
we use the last definition as a convenient way to avoid the cumbersome notation $B \otimes \calk$.  



\begin{definition}
{\rm
Let $(\cale,U)$ be a $G$-Hilbert $A$-module.
An operator $T$ in $\call(\cale)$ is called
{\em $G$-invariant}
if $T$ commutes with 
the operator $U_g: \cale \rightarrow \cale$
(that is, $T \circ U_g = U_g \circ T$) 
for all $g \in G$. (Equivalently: $U_g T U_{g^{-1}} = T U_{g g^{-1}}$.)
}
\end{definition}


Note that then $T^*$ automatically commutes also with $U_g$.



\begin{lemma}   \label{lemmainviso}
Let $(B,\beta)$ be a stable $G$-algebra.
%
%
Then there exist 
$G$-invariant isometries $V_1$ and $V_2$ in $\call_B(B)=\calm(B)$
such that $V_1 V_1^* + V_2 V_2^* = 1$.
%
%
%
%
%
\end{lemma}

\begin{proof}
Let $\calk$ be the compact operators on $\ell^2(\N)$.
Write $e_{i,j} \in \calk$ for the standard matrix units.
We define the operators $V_1,V_2 \in \call_{B \otimes \calk} (B \otimes \calk)$ for example by
$$V_k(b \otimes e_{i,j}) = b \otimes e_{2 i+k-2,j}$$
for all $i,j \in \N$, $b \in B$ and $k=1,2$.
%
%
\end{proof}

\begin{lemma}    \label{lemmaUconnected}
Let $(B,\beta)$ be stable.
A $G$-invariant unitary $U \in \calm(B)$ can be connected to $1 \in \calm(B)$
by a $G$-invariant, strictly continuous unitary path in $\calm(B)$.
\end{lemma}

\begin{proof}
In the non-equivariant case this is for example \cite[Lemma 1.3.7]{jensenthomsen}.  
Its canonical proof works equivariantly without modification.
\end{proof}


Let us point out that we have all the necessary techniques for inverse semigroup
equivariant $KK$-theory that we shall need available: $KK^G$ allows an associative Kasparov product, and
$KK^G(A,B)$ is functorial in $A$ and $B$ (see \cite{burgiSemimultiKK}). 

{From now on all $C^*$-algebras are assumed to be trivially graded and separable!}



%
%

\section{The unitization of a $G$-algebra}

We shall later need a unitization of a $G$-algebra.
To this end we cannot simply add a single unit but need to adjoin
the whole $G$-algebra $C^*(E)$
to 
a given $G$-algebra. 
This section is dedicated to describe this.

\begin{definition}
{\rm
Let $C^*(E)$ denote the universal {\em abelian} $C^*$-algebra generated by the
free set $E$ of commuting self-adjoint projections.
This algebra is endowed with the $G$-action $\tau$ induced by $\tau_g(e):= g e g^{-1} \in E$
for $g \in G, e \in E$
which turns it to a $G$-algebra.
} 
\end{definition}

\begin{lemma}   \label{lemmaunitization}
Let $(A,\alpha)$ be a $G$-algebra.
Given $a \in A$ and $z \in C^*(E)$ write
$$a z := z a := \gamma_z(a),$$
where $\gamma:C^*(E) \rightarrow \calm(A)$ denotes the canonical $*$-homomorphism such that $\gamma_e(a)= \alpha_e(a)$
for all $e \in E, a \in A$.

Then the linear direct sum $A \oplus C^*(E)$ 
turns to a $G$-algebra under the operations 
\begin{eqnarray*}
(a \oplus z)^* &:=& a^* \oplus z^*,  \\  
(a \oplus z)\cdot(b \oplus w) &:=& ab + z b + a w \oplus z w   
\end{eqnarray*}
for all $a,b \in A, z,w \in C^*(E)$
and under the diagonal $G$-action
$\alpha \oplus \tau$.
%
\end{lemma}

\begin{proof}
In this proof, $\oplus$ indicates only a linear sum, and $\oplus^{(C^*)}$ a $C^*$-direct sum.
%
%
Put $Z:= \mbox{span}(E) \subseteq C^*(E)$. Of course, $Z$ is a dense $*$-subalgebra
of $C^*(E)$.
%
%
%

We leave it to the reader to show that $A \oplus C^*(E)$ is a $*$-algebra.
We claim
that $A \oplus Z \subseteq A \oplus C^*(E)$ can be equipped with a $C^*$-norm.
Given a finite subset $F \subseteq E$, write $Z_F \subseteq Z \subseteq C^*(E)$
for the finite-dimensional $*$-subalgebra of $C^*(E)$ generated by $F$. 
It is sufficient 
to define a 
$C^*$-norm for each single $A \oplus Z_F$,
since $A \oplus Z$ is the directed union over all such $*$-algebras $A \oplus Z_F$,
and each direct sum $A \oplus Z_F$ must then be topologically closed as $Z_F$ is finite-dimensional.

Since $Z_F$ is a finite-dimensional commutative $C^*$-algebra,
there is an isomorphism $\psi: \C^n \rightarrow Z_F$ of $C^*$-algebras.
Assume by induction hypothesis for $0 \le k < n$ that $A \oplus \psi(\C^k)
\subseteq A \oplus C^*(E)$ is a $C^*$-algebra.
%
%
Let $z = \psi(e_{k+1})$, where $e_{k+1}=0 \oplus \cdots \oplus 0 \oplus 1_{\C} \in \C^{k+1}$.
Multiplication in $A \oplus \psi(\C^{k+1})$
is as follows:
$$(x + \lambda z) (y + \mu z) = x y + \lambda z y + \mu x z + \lambda \mu z$$
for $x,y \in A \oplus \psi(\C^k) \subseteq A \oplus \psi(\C^{k+1})$ and $\lambda,\mu \in \C$.
If $\gamma_z \in A$ then
$$\varphi: A \oplus \psi(\C^{k+1}) \rightarrow (A \oplus \psi(\C^{k}) )
\oplus^{(C^*)} \C:
\varphi(x + \lambda z) = x + \lambda \gamma_z \oplus \lambda$$
defines a $*$-isomorphism ($x \in A \oplus \psi(\C^k), \lambda \in \C$),  
whence $A \oplus \psi(\C^{k+1})$ is a $C^*$-algebra.

If $\gamma_z \notin A$ then consider the operator $P \in \calm(A \oplus \psi(\C^k))$
defined by $P(x)=x z$.
Observe that $P \notin A \oplus \C^k$ (as otherwise $P$ and so $\gamma_z$ would be in $A$,
since $\psi(\C^k) z = 0$).
The injective $*$-homomorphism
$$\varphi:A \oplus \psi(\C^{k+1}) \rightarrow \calm(A \oplus \psi(\C^k)):
\varphi(x + \lambda z) = x + \lambda P$$
($x \in A \oplus \psi(\C^k), \lambda \in \C$)
shows that $A \oplus \psi(\C^{k+1})$ is a $C^*$-algebra.
This completes induction. Thus $A \oplus \psi(\C^n) \cong A \oplus Z_F$ is a $C^*$-algebra.

Note that the canonical projection $A \oplus Z_F \rightarrow Z_F \subseteq C^*(E)$ is a 
$*$-homomorphism of $C^*$-algebras and thus contractive.
Thus by taking the direct limit we obtain a canonical contractive projection $\overline{A \oplus Z} \rightarrow C^*(E)$.
Since $A \oplus C^*(E) \subseteq \overline{A \oplus Z}$, we have a contractive projection $A \oplus C^*(E) \rightarrow C^*(E)$. 
By a standard result in the theory of topological vector spaces,  $A \oplus C^*(E)$ is complete. 
Thus $A \oplus C^*(E) = \overline{A \oplus Z}$ is a $C^*$-algebra.
%

Finally, a straightforward check 
shows that $A \oplus C^*(E)$ is a $G$-algebra.
%
%
\end{proof}

\begin{definition}   \label{defunitization}
{\rm
For a $G$-algebra $(A,\alpha)$ we define its {\em unitization} 
to be the $G$-algebra
$(A^+,\alpha^+) := (A \oplus C^*(E), \alpha \oplus \tau)$
as described in Lemma \ref{lemmaunitization}.
}
\end{definition}

%
Because 
$G$ has a unit, $A^+$ is 
unital.
%
Actually, this 
is the only 
reason and place where we need a unit in $G$.

\section{The split-exactness, stability and homotopy invariance of $KK^G$}

\label{sectionProp}



The aim of this section is the proof of Proposition \ref{propositionFunctor}.

\begin{lemma}   \label{lemmaexa}
Let $D$ be a $G$-algebra and
$$\xymatrix{0 \ar[r] & A  \ar[r]^j & B \ar[r]^f & C \ar[r] & 0}$$
an exact sequence of $G$-algebras.
If $[\varphi,\cale,T] \in KK^G(D,B)$ such that $f_*(\varphi,\cale,T)$ is a degenerate Kasparov cycle in $KK^G(D,C)$, then
it is of the form
$$[\varphi, \cale,T] = j_*[\varphi',\cale',T'],$$
where
$[\varphi', \cale' ,T'] \in KK^G(D,A)$ with $\cale' = \{ \xi \in \cale|\, \langle \xi, \xi \rangle \in j(A)\} \subseteq \cale$,
$\varphi' = \varphi(\cdot)|_{\cale'}$
and $T'= T|_{\cale'}$.
%

\end{lemma}

\begin{proof}
The canonical proof of \cite[Lemma 3.2]{skandalis} works verbatim also $G$-equivariantly.
\end{proof}

\begin{definition}
{\rm
We recall that
$F(-)= KK^G(A,-)$  denotes the functor $F :{\bf C^*} \rightarrow {\bf Ab}$ with $F(B) = KK^G(A,B)$
and $F(f) = f_*(z) = z \otimes_B f_*(1_B)$ for $f \in {\bf C^*}(B,C)$ and $z \in KK^G(A,B)$.
}
\end{definition}

\begin{lemma}    \label{lemmaKstable}
The functor $F$ given by $F(B) = KK^G(A,B)$ from $ {\bf C^*}$ to the abelian groups is stable.
That is, for any $G$-algebra $(B \otimes \calk,\gamma)$
and any minimal projection $e \in \calk$ 
such that the associated corner embedding
$\varphi: B \rightarrow B \otimes \calk$ with
$\varphi(b) = b \otimes e$ happens to be a $G$-equivariant $*$-homomorphism,
the map $F(\varphi)$ is invertible.
\end{lemma}

\begin{proof}
Notice that $F(\varphi)= (\cdot) \otimes_B [\varphi]$, where
$$[\varphi]:= \varphi_*(1_B) = [\mbox{id}_B ,B \otimes_\varphi B \otimes \calk,0]
= [\mbox{id}_B, B \otimes e \calk,0]  \in KK^G(B , B\otimes \calk),$$
where the $G$-action on $B \otimes e \calk$ is given by $\gamma$.
We propose an inverse element for $[\varphi]$ by
$$z := [m,B \otimes \calk e,0] \in KK^G(B \otimes \calk,B),$$
where $m$ is the multiplication operator, the $G$-action on $B \otimes \calk e$ is
given by $\gamma$,
and the $B$-valued inner product on $B \otimes \calk e$ is
defined by $\langle x,y \rangle = \varphi^{-1}(x^*y)$.
We are done when showing that
$[\varphi] \otimes_{B \otimes \calk} z = 1_B$ and $z \otimes_B [\varphi] = 1_{B \otimes \calk}$ in $KK^G$.
We have
$$[\varphi] \otimes_{B \otimes \calk} z = [\mbox{id}_B, (B \otimes e \calk) \otimes_m (B \otimes \calk e) ,0]
= [\mbox{id}_B,B,0] = 1_B$$
by the $G$-Hilbert $B,B$-module isomorphism determined by
$$b_1 \otimes e k_1 \otimes b_2 \otimes k_2 e \mapsto \varphi^{-1}(b_1 b_2 \otimes e k_1 k_2 e)$$ for all $b_i \in B, k_i \in \calk$.
Similarly, $z \otimes_B [\varphi] = 1_{B \otimes \calk}$ is computed by
the $G$-Hilbert $B \otimes \calk, B \otimes \calk$-module isomorphism
given by
$$b_1 \otimes k_1 e \otimes b_2 \otimes e k_2 \mapsto b_1 b_2 \otimes k_1 e k_2.$$
\end{proof}

\begin{lemma}      \label{lemmaKsplitexact}
The functor $F$ given by $F(B) = KK^G(D,B)$ from ${\bf C^*}$ to the abelian groups is split-exact. That is, given a split exact sequence
\begin{displaymath}   
\xymatrix{0 \ar[r] & A \ar[r]^{j} & B \ar@<.5ex>[r]^{f} &
C \ar[r] \ar@<.5ex>[l]^{s} & 0  }
\end{displaymath}
its image under $F$ is canonically split-exact.
\end{lemma}

\begin{proof}
Let $\pi: B \rightarrow \calm(A)= \call_A(A)$ be the 
standard $*$-homomorphism
$\pi(b)(a) = j^{-1}(b j(a))$
associated to the exact sequence
and notice that it is $G$-equivariant.
Consider the $G$-Hilbert $A$-module $A \oplus A$ with grading $\epsilon(x,y)=(x,-y)$.
We have an element $\{j\}^{-1}:=[\varphi,A \oplus A,F] \in KK^G(B,A)$, where
$F$ is the flip automorphism and $\varphi:B \rightarrow \call_A(A \oplus A)$ is given by
$$\varphi(b) (x,y) = (\pi(b)x, (\pi \circ s \circ  f ) (b) y).$$

Similarly, there is an element
$\{s \circ f \}^\bot:=[\psi,B \oplus B,F'] \in KK^G(B,B)$, where $F'$ is the flip automorphism and
$\psi: B \rightarrow \call_B(B \otimes B)$ is determined by
$$\psi(b) (x,y) = ( b x , (s \circ f ) (b) y).$$ 
It was checked in \cite{kasparov1988}
that $\{ s \circ f \}^\bot = 1_B - (s \circ f)_*(1_B)$.
%
Lemma \ref{lemmaexa}
shows that $\{ s \circ f \}^\bot = j_* ( \{j\}^{-1})$.

Let $D$ be another $G$-algebra, and $x \in KK^G(D,B)$ be in the kernel of $f_*$.
Then
\begin{eqnarray*}
x & =&  x - (s \circ f)_*(x) = x \otimes_B (1_B - (s \circ f)_*(1_B)) \\
&=& x \otimes_B j_* ( \{j \}^{-1}) = j_* ( x \otimes_B \{j \}^{-1}),
\end{eqnarray*}
which is in the image of $j_*$.
Thus
the sequence
\begin{displaymath}  
\xymatrix{0 \ar[r] & KK^G(D,A) \ar[r]^{j_*} & KK^G(D,B) \ar@<.5ex>[r]^{f_*} &
KK^G(D,C) \ar[r] \ar@<.5ex>[l]^{s_*} & 0  }
\end{displaymath}
is split exact.
\end{proof}

By Lemmas \ref{lemmaKstable} and \ref{lemmaKsplitexact}
and the 
evidence of homotopy invariance we obtain the 
main result of this section:

\begin{corollary}    \label{corollaryProp}
Proposition \ref{propositionFunctor}
is true.
\end{corollary}

%

\section{Cocycles}

When we shall later introduce a Cuntz picture of Kasparov theory,
the corresponding transformation produces
a $G$-action $S$ on a $G$-Hilbert $A$-module $A$, which
will be - synthetically - written as $S_g = u_g \circ \alpha_g$,
where $\alpha$ denotes the $C^*$-action on $A$.
This $u_g = S_g \circ \alpha_{g^{-1}}$ will be defined next:
%


\begin{definition}
{\rm
Let $(A,\alpha)$ be a $G$-algebra. An 
{\em $\alpha$-cocycle}
is a map $u:G \rightarrow \calm(A)$
such that the identities
\begin{equation} \label{defcocycle}
\alpha_{g g^{-1}}  = u_g^* u_g, \quad	
u_{g g^{-1}} = u_{g} u_{g}^*,
\quad
	u_{g h} = u_g \overline{\alpha_g}(u_h)
\end{equation}
hold in $\calm(A)$ for all $g$ and $h$ in $G$.
}
\end{definition}


\begin{lemma}		\label{lemmacocycle}
Let $u$ be an $\alpha$-cocycle.
\begin{itemize}

\item[(a)]
Then we have
$$\alpha_{g g^{-1}} = u_g^* u_g = u_g u_g^* = u_{g g^{-1}} \in \calm(A).$$
In particular,
every $u_g$ is a partial isometry and
the source and range projection of $u_g$ both agree with $\alpha_{g g^{-1}}$ and are in the center
of $\calm(A)$.

\item[(b)] In particular, every $u_{e} = \alpha_e$ is a self-adjoint projection in the center of $\calm(A)$
for all $e \in E$.

\item[(c)]
We may replace the second identity of (\ref{defcocycle}) by the identity
$$\overline{\alpha_g}(u_{g^{-1}}) = u_g^*$$
without changing the definition of a cocycle.
\end{itemize}

\end{lemma}

\begin{proof}
Note that $\alpha_{g g^{-1}}$ is a projection of the center of $\calm(A)$. Hence, $u_g$ is a partial isometry by the first identity of
(\ref{defcocycle}).
Using only the identities (\ref{defcocycle}), we have
$$\overline{\alpha_g}(u_{g^{-1}}) = u_g^* u_g \overline{\alpha_g}(u_{g^{-1}}) = u_g^* u_{g g^{-1}}
= u_g^* u_{g} u_{g}^* = u_g^*,$$
which checks 
Lemma
\ref{lemmacocycle}.(c).
The second identity of (\ref{defcocycle}) is on the other hand easily obtained from this new identity.
The identity $u_g^* u_g = u_g u_g^*$ follows now from
the first identity of (\ref{defcocycle}) and the identity of Lemma
\ref{lemmacocycle}.(c) through
\begin{eqnarray*}
u_g^* u_g &=& \alpha_{g} \alpha_{g^{-1} g } \alpha_{g^{-1}} = \alpha_g u_{g^{-1}}^* u_{g^{-1}} \alpha_{g^{-1}}
= \alpha_g \overline{\alpha_{g^{-1}}}(u_{g}) u_{g^{-1}} \alpha_{g^{-1}} \\
&=& \alpha_g \circ \alpha_{g^{-1}} \circ u_{g} \circ \alpha_g \circ  u_{g^{-1}} \circ \alpha_{g^{-1}}
= u_g \overline \alpha_g(u_{g^{-1}}) 
=  u_g u_g^* .
\end{eqnarray*}
%
\end{proof}

\begin{definition}
{\rm
Given an $\alpha$-cocycle $u$ we write $u \alpha u^*$ for the $G$-action
$(u \alpha u^*)_g(a) = u_g \alpha_g(a) u_g^*$ on $A$. 
}
\end{definition}

\begin{definition}
{\rm
For an $\alpha$-cocycle $u$ we introduce a $G$-action $\delta^u$ on
$M_2(A)$ under the formula
$$\delta^u_g \left ( \begin{matrix} a & b \\ c & d \end{matrix} \right )
= \left ( \begin{matrix} \alpha_g(a) & \alpha_g(b) u_g^* \\ 
u_g \alpha_g(c) & u_g \alpha_g(d) u_g^* \end{matrix} \right ) . $$
%
}
\end{definition}

Notice that $\alpha_{e}(a) u_{e}^* = \alpha_{e}(a) \alpha_{e}
= \alpha_{e}(a)$ for every $e \in E$ by Lemmas \ref{lemmacocycle}
and \ref{lemmamultiae}, such that
$\delta^u_e = \mbox{id}_{M_2} \otimes \alpha_e$.
With that and Lemma \ref{lemmacocycle} it is straighforward to check that
$\delta^u$ is indeed a $G$-action.

\section{The isomorphism $u_\#$}

In this section we shall see that the objects $(A,\alpha)$ and
$(A, u \alpha u^*)$ are isomorphic under a stable functor, where $u$ denotes an $\alpha$-cocycle.

\begin{definition}
{\rm
Consider the two corner embeddings $S_A(a) = \left ( \begin{matrix} a & 0 \\ 0 & 0 \end{matrix} \right )$ and $T_A(a)=\left ( \begin{matrix} 0 & 0 \\ 0 & a \end{matrix} \right )$
which define $G$-equivariant $*$-homomorphisms
$S_A : (A,\alpha) \rightarrow (M_2(A), \delta^u)$
and
$T_A : (A,u \alpha u^*) \rightarrow (M_2(A), \delta^u)$.
}
\end{definition}

\begin{definition}
{\rm
Let $F$ be a stable functor from ${\bf C^*}$ to the abelian groups.
Then define an isomorphism
$$u_\# := F(T_A)^{-1} \circ F(S_A) : F(A,\alpha) \rightarrow F(A, u \alpha u^*).$$
}
\end{definition}

That is, under a stable functor `the actions $\alpha$ and $u \alpha u^*$
are isomorphic' and we can switch between them via $u_\#$ as we like.

\begin{lemma}
Consider the stable functor $F$ from ${\bf C^*}$ to the abelian groups defined by
$F(A)= KK^G(D,A)$. Then the map $u_\#$ from the last definition and its inverse map $u_\#^{-1}$
can be realized by multiplication with the following Kasparov elements:
\begin{eqnarray*}
z &:= &(\mbox{id}_A, (A,\alpha u^*), 0) \in KK^G \big((A,\alpha), (A,u \alpha u^*)\big ),\\
z^{-1} &:=& (\mbox{id}_A, (A,u \alpha), 0) \in KK^G \big((A,u \alpha u^*), (A,\alpha) \big ),
\end{eqnarray*}
respectively, where the occurring Hilbert $A$-modules are trivially graded
and $(\alpha u)_g(a) = \alpha_g(a) u_g$ and $(\alpha u^*)_g(a) = \alpha_g(a) u_g^*$
denote their $G$-actions, respectively.

\end{lemma}

\begin{proof}
Since $u_\#^{-1} = (S_A)_*^{-1} \circ (T_A)_*$ the claim is that
$z^{-1} = [T_A] \otimes_{M_2(A)} [S_A]^{-1}$.
%
The $KK^G$-inverse $[S_A]^{-1}$ may be represented by 
$$[m, M_2(A) \left ( \begin{matrix} 1 & 0 \\ 0 & 0 \end{matrix} \right ),0 ]
\in KK^G(M_2(A),A),$$
where $m$ denotes the multiplication operator.
On the other hand $[T_A] = [ T_A,\left ( \begin{matrix} 0 & 0 \\ 0 & 1 \end{matrix} \right )
M_2(A),0  ] \in KK^G(A,M_2(A))$. Here, the Hilbert modules have trivial grading and the $G$-actions are given by restriction
of $\delta^u$.
We have an isomorphism
$$\left ( \begin{matrix} 0 & 0 \\ 0 & 1 \end{matrix} \right )
M_2(A)
\otimes_{M_2(A)}
M_2(A) \left ( \begin{matrix} 1 & 0 \\ 0 & 0 \end{matrix} \right )
\rightarrow \left ( \begin{matrix} 0 & 0 \\ A & 0 \end{matrix} \right ):
x \otimes y \mapsto xy$$
of $G$-Hilbert $A,A$-bimodules, where the image is also equipped with the restricted $\delta^u$-action.
This proves the claim. The case $u_\#$ is proven similarly.
\end{proof}

\begin{lemma}   \label{lemmakreuz}
Let $\varphi:(A,\alpha) \rightarrow (B,\beta)$
be an equivariant $*$-homomorphism.  
Let $u$ be an $\alpha$-cocycle
and $v$ a $\beta$-cocycle
such that $v_g \varphi(a) = \varphi( u_g a)$ for all $g \in  G$ and $a \in A$.
Let $F$ be a stable functor from ${\bf C^*}$ to ${\bf Ab}$.
Then 	
$\varphi$ is also an equivariant $*$-homomorphism 
$\varphi:(A,u \alpha u^*) \rightarrow (B, v \beta v^*)$
such that 
$$v_\# \circ F(\varphi)  = F(\varphi) \circ u_\#
\; : \;F(A,\alpha) \rightarrow F(B,v \beta
v^*).$$
\end{lemma}

\begin{proof}
Just note that $\mbox{id}_{M_2} \otimes \varphi:(M_2(A),\delta^u) \rightarrow
(M_2(B),\delta^v)$ is an equivariant $*$-homomorphism satisfying
$(\mbox{id}_{M_2} \otimes \varphi) \circ S_A = S_B \circ \varphi$
and
$(\mbox{id}_{M_2} \otimes \varphi) \circ T_A = T_B \circ \varphi$.
\end{proof}

\section{The cocyle set $\E^G(A,B)$}

Until Section \ref{sectionPsiprime}
assume that $B$ is stable (i.e. $B \cong B \otimes \calk$)!

The $A,B$-cocycles defined next will serve as a Cuntz-picture of Kasparov
cycles. We shall prove in the next section that they may substitute Kasparov
theory.
Confer Remark \ref{remarkcocycles} for a motivation of the following definition:

\begin{definition}  \label{defKaspCocycle}
{\rm
Let $(A,\alpha)$ and $(B,\beta)$ be $G$-algebras (where $B$ is stable).
An {\em 
$A,B$-cocycle}
is a quadruple
$$(\varphi_\pm,u_\pm):=(\varphi_+,\varphi_-,u_+,u_-)
\in \big(\mbox{Hom}(A,\calm(B)) \big)^2 \times \big(\calm(B)^G
\big)^2
$$
where
\begin{itemize}
\item[(a)]
$u_+$ and
$u_-$ denote $\beta$-cocycles
$G \rightarrow \calm(B)$,

\item[(b)]
$\varphi_\pm$ denote $G$-equivariant $*$-homomorphisms
$A \rightarrow \big (\calm(B),u_\pm \overline \beta u_\pm^* \big)$, respectively,

\item[(c)] $\varphi_+(a) - \varphi_-(a) \in B$,

\item[(d)] ${u_+}_g - {u_-}_g \in B$
\end{itemize}
for all $a$ in $A$ and $g$ in $G$.
}
\end{definition}

\begin{definition}
{\rm
Two $A,B$-cocycles $(\varphi_\pm,u_\pm)$ and $(\psi_\pm,v_\pm)$ are {\em isomorphic}
when there exists a $G$-equivariant automorphism $\gamma \in \mbox{Aut}(B,\beta)$ such that
$$\big (\overline \gamma \circ \varphi_\pm, \overline \gamma(u_\pm) \big) = (\psi_\pm,v_\pm).$$
}
\end{definition}

In the rest of the paper we shall identify isomorphic
$A,B$-cocycles.

\begin{definition}
{\rm
The set of isomorphism classes of 
$A,B$-cocycles is denoted by
$\E^G(A,B)$.
}
\end{definition}

\begin{definition}
{\rm
An $A,B$-cocycles $(\varphi_\pm,u_\pm)$ is called {\em degenerate}
if $\varphi_\pm=0$. The set of degenerate $A,B$-cocycles is denoted by
$\D^G(A,B)$.
}
\end{definition}

\begin{definition}     \label{defhomotopycocycles}
{\rm
Two $A,B$-cocycles $(\varphi_\pm^{t},u_\pm^{t})$ ($t=0,1$) 
are called
{\em homotopic} if there exists
an $A,B[0,1]$-cocycle $(\varphi_\pm,u_\pm)$ such that
$$\big (\overline{\pi_t} \circ \varphi_\pm, \overline{\pi_t} (u_\pm) \big)
= \big (\varphi_\pm^{t},u_\pm^{t} \big),$$
where $\pi_t : B[0,1] \rightarrow B$ denotes evaluation at time $t=0,1$.
}
\end{definition}

\begin{definition}     \label{defsumcocycles}
{\rm
For $(\varphi_\pm,u_\pm),(\psi_\pm,v_\pm) \in \E^G(A,B)$ define their {\em sum} 
to be
$$(\varphi_\pm,u_\pm) + (\psi_\pm,v_\pm) :=
(V_1 \varphi_\pm V_1^* + V_2 \psi_\pm V_2^*, \, V_1 u_\pm V_1^* + V_2 v_\pm V_2^*)
\,  \in \E^G(A,B),$$
where $V_1,V_2 \in \calm(B)$ are $G$-invariant isometries such that
$V_1 V_1^* + V_2 V_2^* = 1$ (see Lemma \ref{lemmainviso}).
}
\end{definition}

\begin{lemma}
Up to homotopy of $A,B$-cocycles, the last definition of sum of $A,B$-cocycles does
not depend on the choice of the isometries $V_1,V_2$.

\end{lemma}

\begin{proof}
Let $W_1,W_2 \in \calm(B)$ be another pair of $G$-invariant isometries such that $W_1 W_1^* + W_2 W_2^* = 1$.
Then $U= W_1 V_1^* + W_2 V_2^*$ defines a $G$-invariant unitary in $\calm(B)$ such that $U V_i = W_i$ ($i=0,1$). 
By Lemma \ref{lemmaUconnected}, $U$ may be connected to $1$ by a $G$-invariant, strictly continuous path $(U_t)_t$ in $\calm(B)$.
Then the cocycle
$$(U_t V_1 \varphi_\pm V_1^* U^*_t +
U_t V_2 \psi_\pm V_2^* U^*_t ,  \, U_t V_1 u_\pm V_1^* U_t^* + U_t V_2 v_\pm V_2^* U_t^*)_{t \in [0,1]}$$
in $\E^G(A,B[0,1])$
yields the desired homotopy.
\end{proof}

\begin{definition}    \label{defFG}
{\rm
Let $\F^G(A,B)$ denote the quotient $\E^G(A,B)/\sim$ under the equivalence relation
$\sim$ defined by $x_1 \sim x_2$ for $x_1,x_2\in \E^G(A,B)$ if and only if there exists
degenerate $d_1,d_2 \in \D^G(A,B)$ such that $x_1 + d_1$ is homotopic to $x_2+d_2$.
We equip $\F^G(A,B)$ with the addition $[x_1]+ [x_2]:= [x_1+x_2]$.
}
\end{definition}

\section{The isomorphism $\Phi$}

\label{sectionPhi}

In this section we shall 
isomorphically substitute Kasparov theory
by its Cuntz picture in form of $A,B$-cocycles.
This transition is given as follows:

\begin{definition}   \label{defPhi}
{\rm
There is a map $$\Delta:\E^G(A,B) \rightarrow KK^G(A,B)
, \quad\Delta(\varphi_\pm,u_\pm) := [\varphi,B \oplus B,F],$$ where
the $G$-Hilbert $B$-module $B \oplus B$ is equipped with the grading
$\epsilon(x,y)=(x,-y)$ and the $G$-action
$$(u_+ \beta, u_- \beta)_g (x,y) = ({u_+}_g \beta_g(x), {u_-}_g \beta_g(y)),$$
$F$ is the flip automorphism,
and the operator $\varphi:A \rightarrow \call_B(B \oplus B)$
is defined by 
$$\varphi(a)(x,y)=(\varphi_+(a) x, \varphi_-(a) y).$$
}
\end{definition}


\begin{lemma}  \label{lempsi}
The just defined $\Delta(\varphi_\pm,u_\pm)$ is indeed a Kasparov
group element. 
\end{lemma}

\begin{proof}
Denoting $u=(u_+,u_-)$, 
$\gamma =(\beta,\beta)$
and the indicated $G$-action $(u_+\beta,u_- \beta)$ on $B \oplus B$ by $W=u \gamma$, one has
%
%
$$W_{gh} = u_{gh} \gamma_{gh} = u_g \overline{\gamma_g}(u_h) \gamma_{g h}
= u_g \gamma_g u_h \gamma_{g^{-1}} \gamma_g \gamma_h = W_g W_h$$
because of the third identity of (\ref{defcocycle}) and because
$\beta_{g^{-1} g}$ is in the center of $\calm(B)$.
We have
$$
\langle W_g (x,y),W_g(a,b)\rangle  = \langle \gamma_g (x,y),u_g^* u_g \gamma_g(a,b)\rangle
= \gamma_g \langle(x,y),(a,b)\rangle,
$$
%
because $\gamma_{g g^{-1}} = u_g^* u_g$ by the first identity
of the cocycle axioms (\ref{defcocycle}).
%
We have $\overline{\gamma_g}(u_{g^{-1}})= u_g^*$ by Lemma \ref{lemmacocycle}.(c),
and thus
$W_g W_{g^{-1}} = u_g \gamma_g u_{g^{-1}} \gamma_{g^{-1}}
= u_g u_g^*$, which is a self-adjoint projection in $\call_B(B \oplus B)$.
By a similar argument, and with condition (b) of Definition \ref{defKaspCocycle}
we get
$$\varphi(\alpha_g(a)) = u_g \overline{\gamma_g}(\varphi(a)) u_g^*
= u_g \gamma_g \varphi(a) \gamma_{g^{-1}} \gamma_g u_{g^{-1}}\gamma_{g^{-1}}
= W_g \varphi_g(a) W_{g^{-1}}.$$

The $B$-module 
multiplication
on $B \oplus B$ is $E$-compatible, in other words
\begin{eqnarray*}
W_e(\xi) b = u_e \gamma_e(\xi) b= \gamma_e(\xi) b = \xi \gamma_e(b)
\end{eqnarray*}
for $\xi \in B \oplus B, b \in B$ and $e \in E$, because $\gamma_{e} = u_e$
by 
Lemma \ref{lemmacocycle}.
%
A straightforward computation shows that the operator $W_g F W_{g^{-1}} - W_g W_{g^{-1}} F$
is in $B \oplus B$ because
${u_+}_g - {u_-}_g$ is in the ideal $B$ by Definition \ref{defKaspCocycle}.
This verifies Definition \ref{defCycle} of a Kasparov cycle.
\end{proof}

%

%

\begin{proposition}    \label{propcycle}
Every element 
of $KK^G(A,B)$
may be represented in the form 
$[\varphi,B \oplus B,\left ( \begin{matrix} 0 & 1 \\ 1 & 0 \end{matrix} \right )]$
for a certain 
$G$-action $S=(S_+ ,S_-)$ on the Hilbert $B$-module $B \oplus B$
with grading $\epsilon(x,y) = (x,-y)$, where $S_\pm$ are $G$-actions on the ungraded
Hilbert $B$-module $B$. Moreover,
\begin{equation}    \label{equkas}
S_g \left ( \begin{matrix} 0 & 1 \\ 1 & 0 \end{matrix} \right ) S_{g^{-1}}
- \left ( \begin{matrix} 0 & 1 \\ 1 & 0 \end{matrix} \right ) S_g S_{g^{-1}} \in
\calk(\call_B(B\oplus B)) \cong M_2(B)
\end{equation}
for all $g \in G$.
\end{proposition}

\begin{proof}
Let $(\varphi,\cale,T) \in KK^G(A,B)$ be given. 
Denote the $G$-action on $\cale$ by $U$. We may assume that
$U_g T U_{g^{-1}} - U_{g g^{-1}} T \in \calk(\cale)$.
(If $G$ were a group then this would be by
Remark 2 on page 156 of Kasparov's paper \cite{kasparov1988}.
But this works also
in our setting by a similar proof as suggested by Kasparov but with
$\varphi(g) = U_g T U_{g^{-1}} - T U_{g g^{-1}}$ rather than $\varphi(g) = U_g T U_g - T$,
and applied to the technical Theorem 1 in \cite{burgiSemimultiKK}
rather than the technical Theorem 1.4 in \cite{kasparov1988}.)

Let $B_2$ denote the $G$-Hilbert $B$-module $B \oplus B$ with grading
$\epsilon(x,y)=(x,-y)$ and $G$-action $(\beta,\beta)$.
Since $(0,B_2,0) \in KK^G(A,B)$ is degenerate,
$$(\varphi,\cale,T) \oplus (0,B_2,0)$$
is homotopic in $KK^G$-theory to $(\varphi,\cale,T)$.
By Kasparov's stabilization theorem (the graded version), there is an isomorphism
$\Lambda:\cale \oplus B_2 \rightarrow B_2$ of graded Hilbert $B$-modules;
we use here the fact that $B$ is stable,
and thus $H_B \cong B$, see \cite[Lemma 1.3.2]{jensenthomsen}. 
We define the $G$-action on $B_2$ in the image of $\Lambda$ in such a way that $\Lambda$
becomes $G$-equivariant, and denote this new $B_2$ by $B_2'$. Hence we may write $[\varphi,\cale,T]=[\psi,B_2',T_1]$.

Since the $G$-action $W$ on $B_2'$ is grading preserving, it must be of the form
$$W_g(x,y)= (S_g x, V_g y),$$
where $S$ and $V$ are $G$-actions on the ungraded homogeneous parts
($B$-parts) of $B_2'$.
Hence
$$W_g'(x,y)= (V_g x, S_g y)$$
is another $G$-action on $B_2$, and we denote this new $B_2$ by $B_2''$.
Using $[\psi,B_2',T_1] = [\psi,B_2',T_1] + [0,B_2'',0]$ (degenerate), 
and using isomorphisms $B \oplus B \cong B$ on the respective homogeneous parts by Kasparov's stabilization theorem,
we may assume that the $G$-action on $B_2'$ is of the form
$$W_g(x,y) = (S_g x,S_g y),$$
where $S$ is a $G$-action on the homogeneous $B$-parts of $B_2'$.

Identifying $\call_B(B_2') \cong M_2(\call_B(B))$, $T_1$ takes on the form
$T_1= \left ( \begin{matrix} 0 & x \\ y & 0 \end{matrix} \right )$.

By considering the same homotopies as in the non-equivariant case,
see \cite[p. 125]{jensenthomsen} (notice that $U_g T_1^n U_{g^{-1}} = (U_g T_1 U_{g^{-1}})^n = T_1^n U_g U_{g^{-1}}$
by Lemma \ref{lemmacenterU} and (\ref{UT})),
we may assume that
$x=y^*$ and $\|x\|\le  1$.
Also, by adding on the degenerate cycle $[0,B_2',0]$, and performing 
the same homotopy as in the non-equivariant case, 
see \cite[p. 126]{jensenthomsen},
we may assume that
$T_1= \left ( \begin{matrix} 0 & u \\ u^* & 0 \end{matrix} \right )$
for some unitary $u \in \calm(B)$.

Define an automorphism $\Theta:B_2 \rightarrow B_2$ of Hilbert $B$-modules
by 
$$\Theta(x,y) = (u^* x,y),$$
and define a $G$-action on its image $B_2$, then denoted by $B_2'''$,
in such a way that $\Theta: B_2' \rightarrow B_2'''$ becomes $G$-equivariant.
Hence $[\psi,B_2',T_1] = [\vartheta,B_2''',\left ( \begin{matrix} 0 & 1 \\ 1 & 0 \end{matrix} \right )]$.
\end{proof}

\begin{remark}  \label{remarkcocycles}
{\rm
We use the last proposition as a basis for a Cuntz picture of $KK$-theory.
The $S_+$-action appearing there we shall define (in the next theorem) to be written as ${S_+}_g = {u_+}_g \circ \beta_{g}$
(exactly the $G$-Hilbert module action appearing in Definition \ref{defPhi} and in Definition \ref{defKaspCocycle}.(b)),
or in other words, we define ${S_+}_g \circ \beta_{g^{-1}} =: {u_+}_g$, and $u_+$ turns out to be a $\beta$-cocycle.
In other words, $u_+$ encodes the difference between the $C^*$-action $\beta$ on $B$ and the Hilbert module action $S_+$ on $B$. 
That is the function of $\beta$-cocycles.
}
\end{remark}

\begin{theorem}  \label{thmisochi}
The set $\F^G(A,B)$ is an abelian group,
and
the map $\Delta$ of Definition \ref{defPhi} canonically induces an abelian group isomorphism
$$\Phi: \F^G(A,B) \rightarrow  KK^G(A,B)$$
by $\Phi([x]):= \Delta(x)$.
%
\end{theorem}

\begin{proof}
It is clear that $\Phi$ is a well-defined map which preserves addition.
It will thus be sufficient to show that $\Phi$ is bijective.

We are going to show that $\Phi$ is surjective.
Let us be given an element $z$ in $KK^G(A,B)$ as indicated in Proposition \ref{propcycle}.
Since $\varphi$ respects grading, it is of the form $\varphi=(\varphi_+,\varphi_-)$.
We claim that $\Phi([\varphi_\pm,u_\pm]) = z$,
where the $\beta$-cocylces $u_\pm$ are defined by ${u_\pm}_g := {S_\pm}_g \circ \beta_{g^{-1}}$.

To check that $u_+$ (and similarly $u_-$) is a
$\beta$-cocycle, we compute
\begin{eqnarray*}
\langle {u_+}_g x,y \rangle &=& \langle {S_+}_g \beta_{g^{-1}} x,y \rangle
= \langle {S_+}_g \beta_{g^{-1}} x, {S_+}_g {S_+}_{g^{-1}} y \rangle
= \beta_g (\langle \beta_{g^{-1}} x, {S_+}_{g^{-1}} y \rangle )\\
&=&  \beta_g \beta_{g^{-1}} (x^*) \cdot \beta_g {S_+}_{g^{-1}} (y)
= x^* \cdot  \beta_g {S_+}_{g^{-1}} (y) = \langle x, \beta_{g} {S_+}_{g^{-1}} y \rangle
\end{eqnarray*}
for all $x$ and $y$ in $B$ by Lemma \ref{lemmacenterU}
and Definitions \ref{defCstar} and \ref{defHilbert}, so that
\begin{equation}    \label{ustar}
{{u_\pm}_g}^* = \beta_{g} \circ {S_\pm}_{g^{-1}}.
\end{equation}
For an idempotent $e \in E$ and all $x,y \in B$
we have
${S_+}_e(x) y = x \beta_e(y) = \beta_e(x) y$  by Definitions
\ref{defCstar} and \ref{defHilbert},
so that we obtain
\begin{equation}   \label{sequb}
{S_+}_e = \beta_e.
\end{equation}
This shows that
$${{u_+}_g}^* {u_+}_g = \beta_{g} {S_+}_{g^{-1}} {S_+}_g \beta_{g^{-1}}
= \beta_{g} \beta_{g^{-1} g} \beta_{g^{-1}} = \beta_{g} \beta_{g^{-1}},$$
the first identity of (\ref{defcocycle}). 
Similarly we get the second identity and the third one computes as
$${u_+}_{g h} = {S_+}_{g h} \beta_{ h^{-1} g^{-1}}
= {S_+}_g \beta_{g^{-1}} \beta_g {S_+}_h \beta_{{h}^{-1}} \beta_{g^{-1}}
= {u_+}_g \overline \beta_g ({u_+}_h).$$
 
%
%

We note that,
since ${S_+}_{g^{-1} g} = \beta_{g^{-1} g}$, we have, with $S:=(S_+,S_-)$, (\ref{sequb})
and (\ref{equkas}),
\begin{eqnarray*}
&&S_g \left (\begin{matrix}
0 & 1\\ 1 & 0
\end{matrix} \right) S_{g^{-1}} -
\left (\begin{matrix}
0 & 1\\ 1 & 0
\end{matrix} \right) S_g S_{g^{-1}}\\
&=& \left( \begin{matrix}
0 & {S_+}_g \beta_{g^{-1} g} {S_-}_{g^{-1}} - {S_-}_{g} \beta_{g^{-1} g} {S_-}_{g^{-1}}\\
x & 0
\end{matrix} \right)
=
\left( \begin{matrix}
0 & ({u_+}_g  - {u_-}_g) {{u_-}_g}^*\\
x & 0
\end{matrix} \right)
\end{eqnarray*}
is in $M_2(B)$
for a certain obvious but irrelevant $x$, and thus
$$
({u_+}_g  - {u_-}_g) {{u_-}_g}^* {{u_-}_g} = ({u_+}_g  - {u_-}_g)
\beta_{g g^{-1}}
={u_+}_g  - {u_-}_g$$
is in $B$ as required
by item (d) of Definition \ref{defKaspCocycle}.


Since $\varphi$ is $G$-equivariant, we have $\varphi_\pm(\alpha_g(a))= {S_\pm}_g \varphi_\pm(a)
{{S_\pm}_{g^{-1}}}$.
Thus, by (\ref{ustar}) and (\ref{sequb}),
\begin{eqnarray*}
&& ({u_+} \overline \beta  {u_+}^*)_g \big (\varphi_+(a) \big ) =
{u_+}_g \overline \beta_g  \big(\varphi_+(a) \big)   {u_+}^*_g  \\
&=& {S_+}_g \circ \beta_{g^{-1}} \circ \beta_g \circ \varphi_+(a) \circ \beta_{g^{-1}} 
\circ \beta_{g} \circ {S_+}_{g^{-1}}
= {S_+}_g \circ  \varphi_+(a) \circ {S_+}_{g^{-1}}  \\
&=& \varphi_+(\alpha_g(a)),
\end{eqnarray*}
which verifies item (b) of Definition \ref{defKaspCocycle}.


%

Now notice that
indeed $\Delta(\varphi_\pm,u_\pm) = z$ (see Definition \ref{defPhi}), since $u_\pm \beta = {S_\pm}$
by (\ref{sequb}).
This proves surjectivity of $\Phi$.


We are going to prove injectivity of $\Phi$.
Let $(\varphi^i_\pm,u_\pm^i) \in \E^G(A,B)$ for $i=0,1$. Assume that $\Phi([\varphi^0_\pm,u_\pm^0]) =
\Phi([\varphi^1_\pm,u_\pm^1])$. Then there exists a Kasparov cycle $(\sigma,\cale,T)$ in $KK^G(A,B[0,1])$
connecting the two cycles $\Delta(\varphi_\pm^i,u_\pm^i)$. 
We apply the procedure described in the surjectivity proof of $\Phi$ (the construction of the preimage of an element)
to the cycle $(\sigma,\cale,T)$, and end up with an element $(\psi_\pm,v_\pm) \in \E^G(A,B[0,1])$.
This is also a homotopy in the sense of Definition \ref{defhomotopycocycles}.

Because at the endpoints of the cycle $(\sigma,\cale,T)$ we have already the nice form
of Definition \ref{defPhi}, all operations that we perform in Proposition \ref{propcycle} for $(\sigma,\cale,T)$
are empty at the endpoints,
except adding on degenerate cycles and application of the Kasparov stabilization theorem. 
Thus, at the endpoints of the 
homotopy $(\psi_\pm,v_\pm)$ we have the following situation.
Let $\pi_t:B[0,1] \rightarrow B$ be the evaluation map for $t \in [0,1]$.
There is a degenerate $A,B$-cocycle $(0,0,z_\pm) \in \D^G(A,B)$
and an 
isomorphism $\Lambda:B \oplus B \rightarrow B$ of Hilbert $B$-modules
such that
\begin{eqnarray*}
\overline \pi_i \circ \psi_\pm (\cdot) &=& \Lambda (\varphi_\pm^i(\cdot) \oplus 0) \Lambda^{-1}  \\
\overline \pi_i({v_\pm}_g) \circ \beta_g  &=&  \Lambda ( {u_\pm^i}_g \circ \beta_g \oplus {z_\pm}_g \circ \beta_g) \Lambda^{-1}
\end{eqnarray*}
for $i=0,1$.

Our next goal is to make $\Lambda$ $G$-equivariant by
multiplying it with some unitary path.  
Define isometries $W_1,W_2 \in \calm(B)$ by $W_1(x)=\Lambda(x,0)$ and $W_2(y)=\Lambda(0,y)$, so that
$$\Lambda(x,y)= W_1(x) + W_2(y)$$
and $W_1 W_1^* + W_2 W_2^* =1$.
Choose $G$-invariant isometries $V_1,V_2 \in \calm(B)$ as in Lemma \ref{lemmainviso}.
Consider the unitary $U= V_1 W_1^* + V_2 W_2^* \in \calm(B)$ and - as the unitary group of $\calm(B)$ is connected
by Lemma \ref{lemmaUconnected} -
connect it to $1 \in \calm(B)$
by a unitary path $(U_t)_{t \in [0,1]}$ in $\calm(B)$.
%
Then
$$ \big ( \; U_t (\overline \pi_i \circ \psi_\pm (\cdot)) U_t^*, \;
U_t \circ \overline \pi_i(v_g) \circ \beta_g \circ U_t^*
\circ \beta_{g^{-1}} \; \big)_{t \in [0,1]}
\in \E^G(A,B[0,1])$$
defines a path of $A,B$-cocycles which connects the two elements
$$\big (\overline \pi_i \circ \psi_\pm,\overline \pi_i(v) \big ), \qquad  \big (\varphi_\pm^i,
u_\pm^i \big ) + \big (0, 0, z_\pm \big )$$
in $\E^G(A,B)$
by Definitions \ref{defhomotopycocycles} and \ref{defsumcocycles}.
Together with the homotopy 
$(\psi_\pm,v_\pm)$ and Definition \ref{defFG} this shows that
$[\varphi_\pm^0, {u_\pm^0}_g] = [\varphi_\pm^1, {u_\pm^1}_g]$.
\end{proof}

\begin{definition}    \label{deffunctorialityE}
{\rm
For 
an equivariant $*$-homomorphism $\lambda:B \rightarrow C$
(where $B$ and $C$ are stable) define an abelian group homomorphism
$$\lambda_* :\F^G(A,B) \rightarrow \F^G(A,C):  \lambda_*[x] = \Phi^{-1}(\lambda_* \Phi([x])) . $$
}
\end{definition}


\begin{lemma}   \label{lemmaphifunct}
Let $\lambda_1: (B_1,\beta_1) \rightarrow (C_1,\gamma_1)$ be a unital $*$-homomorphism of unital $G$-algebras $B_1,C_1$.
Let
$\lambda := \lambda_1 \otimes \mbox{id}: (B,\beta):= (B_1 \otimes \calk, \beta_1 \otimes \mbox{id})
\rightarrow (C,\gamma):= (C_1 \otimes \calk, \gamma_1 \otimes \mbox{id}) $.
%
Then one has
$$\lambda_* [\varphi_\pm,u_\pm] = [\overline
\lambda \circ
\varphi_\pm, \overline \lambda(u_\pm)]. $$
%
\end{lemma}

\begin{proof}
(Sketch)
We have $\overline \lambda( \beta_{g g^{-1}}) = \overline \lambda( \overline \beta_{g} (1))
= \overline \gamma_g \overline \lambda(1) = \gamma_{g g^{-1}}$, so that it is easy to see that
$\overline \lambda(u_\pm)$ are $\gamma$-cocycles.

By unitality of $\lambda_1$, $B \otimes_B C \cong C$ 
as $G$-Hilbert $C$-modules via $x \otimes y \mapsto \lambda(x) y$.
Under this isomorphism, $\varphi_\pm \otimes \mbox{id}_C: A \rightarrow \call_C(B \otimes_B C)$ turns
to $\overline \lambda\circ \varphi_\pm$,    
and the $G$-Hilbert $C$-module actions $u_\pm \beta \otimes \gamma$ on 
$B \otimes_B C$
turn to $\overline \lambda (u_\pm) \gamma$.
%
%
\end{proof}

\section{The map $\Psi$}

\label{sectionPsi}


In this section we shall see how elements of Kasparov theory $KK^G(A,B)$ -
in its form of $A,B$-cocylces (Cuntz picture cycles) in $\E^G(A,B)$ by the isomorphism
$\Phi$ if we like -
induce homomorphisms in $\mbox{Hom}(F(A),F(B))$ for every split exact, homotopy invariant
stable functor $F$ from ${\bf C^*}$ to ${\bf Ab}$.
This goes back to Cuntz in its core, see \cite{cuntz1984}.

\begin{definition}   \label{defAx}
{\rm
Given an $A,B$-cocycle $x=(\varphi_{\pm},u_{\pm}) \in \E^G(A,B)$
we define a $C^*$-algebra
$$A_x  \,:=\,  \{(a,m) \in A \oplus \calm(B)| \, \varphi_+(a) = m \,\, \mbox{modulo} \,\, B \}$$
with two $G$-actions ($+$ and $-$)
$$\Gamma^\pm = (\alpha, u_\pm \overline \beta u_\pm^*).$$
A $\Gamma^+$-cocycle $u$ for $(A_x,\Gamma^+)$ is given by
\begin{displaymath} 
u_g(a,m) = (\alpha_{g g^{-1}}(a), {u_-}_g {u_+}_g^* m )
\end{displaymath}
for $a \in A, m \in \calm(B)$.
}
\end{definition}

\begin{definition}
{\rm
We sloppily use $\Gamma^\pm$ also to denote the $G$-action on $B$
by restricting $\Gamma^\pm$ to $B$, that is, $\Gamma^\pm := u_\pm \overline \beta u_\pm^*$
on $B$.
}
\end{definition}

\begin{lemma}
Definition \ref{defAx} is valid.
\end{lemma}

\begin{proof}
We show that $u$ is a $\Gamma^+$-cocycle.
By Lemma \ref{lemmacocycle}, and since ${u_-}_g^* {u_-}_g$ is in the center
of $\calm(B)$, we have
\begin{eqnarray*}
&& u_g^* u_g(a,m) = (\alpha_{g g^{-1}}(a),{u_+}_g {u_-}_g^* {u_-}_g {u_+}_g^* m ) \\
&=& (\alpha_{g g^{-1}}(a),{u_+}_g {u_+}_g^* {u_-}_g^* {u_-}_g m )
=  (\alpha_{g g^{-1}}(a), \beta_{g g^{-1}} m  \beta_{g g^{-1}} ) \\
&=& (\alpha_{g g^{-1}}(a), {u_+}_{g g^{-1}} \overline \beta_{g g^{-1}}(m) {{u_+}_{g g^{-1}}}^*)
= \Gamma^{+}_{g g^{-1}}(a,m).
\end{eqnarray*}

This shows that $u_g^* u_g = \Gamma^{+}_{g g^{-1}}$, and so the first identity
of (\ref{defcocycle}). The second identity of (\ref{defcocycle}) is left to the reader
and the third one is computed as follows:
\begin{eqnarray*}
&& u_{g} \overline \Gamma_g^+(u_h) (a,m) = u_{g} \circ \Gamma_g^+ \circ u_h \circ
\Gamma_{g^{-1}}^+ (a,m) \\
&=& \big (\alpha_{g g^{-1} g h h^{-1} g^{-1}}  (a), \; {u_-}_g {u_+}_g^* {u_+}_g 
\overline \beta_g \big ( \;
{u_-}_h {u_+}_h^* {u_+}_{g^{-1}} 
\overline \beta_{g^{-1}}  (m)
{u_+}_{g^{-1}}^*
\;
\big ) {u_+}_g^*   \;\big ) \\
&=& \big (\alpha_{g h h^{-1} g^{-1}}  (a), \; {u_-}_{gh}
\overline \beta_g \big ( \; {u_+}_h^* {u_+}_{g^{-1}} 
\overline \beta_{g^{-1}}  (m)
{u_+}_{g^{-1}}^*
\;
\big ) {u_+}_g^*   \;\big )   \\
&=& \big (\alpha_{g h h^{-1} g^{-1}}  (a), \; {u_-}_{gh}
\overline \beta_g( {u_+}_h)^* {u_+}_{g}^* 
\overline \beta_{g g^{-1}}  (m)
{u_+}_{g}
 {u_+}_g^*   \;\big ) \\
&=& \big (\alpha_{g h h^{-1} g^{-1}}  (a), \; {u_-}_{gh}
{u_+}_{gh}^*
(m)
 \;\big ) \\
&=&  u_{gh}(a,m)
\end{eqnarray*}
with the usual center properties, identities (\ref{defcocycle}),
and the identity of Lemma \ref{lemmacocycle}.(c).

We show that $A_x$ is invariant under the $G$-action $u$. 
Let $(a,m) \in A_x$, so $\varphi_+(a) - m \in B$.
By items (b) and (d) of Definition \ref{defKaspCocycle}
and the identity $\beta_{g g^{-1}} = {u_+}_{g g^{-1}}$ of (\ref{defcocycle})
we get modulo $B$
\begin{eqnarray*}
&& \varphi_+(\alpha_{g g^{-1}}(a))
= {u_+}_{g g^{-1}} \overline \beta_{g g^{-1}} \big ( \varphi_+(a) \big )  {u_+}_{g g^{-1}}^*
\equiv \overline \beta_{g g^{-1}} (m)\\
&=& \beta_{g g^{-1}} m = {u_-}_g {u_-}_g^* m \equiv {u_-}_g {u_+}_g^* m.
\end{eqnarray*}
This proves that $u_g(a,m)$ is in $A_x$.
\end{proof}

\begin{definition}    \label{defpsi}
{\rm
Let $x=(\varphi_\pm, u_\pm) \in \E^G(A,B)$.
We have two split exact sequences ($+$ and $-$)
\begin{equation}  \label{splitAx}
\xymatrix{0 \ar[r] & (B,\Gamma^\pm) \ar[r]^{j} & (A_x,\Gamma^\pm) \ar@<.5ex>[r]^{p} &
(A,\alpha) \ar[r] \ar@<.5ex>[l]^{s^\pm} & 0,  }
\end{equation}
where $j(b) = (0,b)$, $p(a,m) = a$ and $s^\pm(a)=(a,\varphi_\pm(a))$.

Let $F$ be a stable, homotopy invariant, split-exact functor from ${\bf C^*}$
to ${\bf Ab}$.
Define
an abelian group homomorphism
$$\Psi_x:F(A,\alpha) \rightarrow F(B,\beta)$$
by
$$\Psi_x = {u_-}_\#^{-1} \circ F(j)^{-1} \circ \big (u_\# \circ F(s^+) - F(s^-) \big).$$
}
\end{definition}

Notice here that the occurrence of $F(j)^{-1}$ is valid as $u_\#$ alters only the $G$-action
whence $F(p) \circ
(u_\# \circ F(s^+) - F(s^-)) = 0$.
Observe that $u_\# \circ F(s^+)$ maps into $(A_x, \Gamma^{-})$.

\begin{lemma}    \label{lemmahomotopy}
The definition of $\Psi_x$ is insensitive against homotopy equivalence of $x$.
\end{lemma}

\begin{proof}
%
Let $x=(\varphi_\pm,u_\pm) \in \E^G(A,B[0,1])$ be a homotopy. Let $\pi_t:B[0,1] \rightarrow B$
be evaluation at time $t=0,1$.
Define two end points
$$x_t :=
\big (\varphi_\pm^{(t)} , u_\pm^{(t)} \big ) := \big(\overline \pi_t \circ \varphi_\pm, \overline \pi_t(u_\pm) \big) \in
\E^G(A,B)$$
($t=0,1$).
The exact sequence (\ref{splitAx}) produces a commutative diagram
\begin{equation}  \label{com1}
\xymatrix{
0 \ar[r] & \big (B,{\Gamma^\pm}^{(t)} \big ) \ar[r]^{j^{(t)}} &
\big (A_{x_t},
{\Gamma^\pm}^{(t)} \big ) \ar@<.5ex>[rr]^{p^{(t)}} &&
(A,\alpha) \ar[r] \ar@<.5ex>[ll]^{{s^\pm}^{(t)}} & 0 \\ 
0 \ar[r] & (B[0,1],\Gamma^\pm) \ar[r]^{j} \ar[u]^{\pi_t} & 
(A_x,\Gamma^\pm) \ar@<.5ex>[rr]^{p} \ar[u]^{\lambda_t}
& & (A,\alpha) \ar[r] \ar@<.5ex>[ll]^{s^\pm} \ar@{=}[u] & 0, \\ 
}
\end{equation}
where 
$\lambda_t:=(\mbox{id}_A, \overline \pi_t)$.
%
%
%
%
%
Note that
$$
\overline \lambda_t (u_g ) = \big(\alpha_{g g^{-1}},
\overline \pi_t (u_- u_+^*) \big ) =
u_g^{(t)},
$$
whence Lemma \ref{lemmakreuz} applies to $\lambda_t$ and the $\Gamma^+$-cocycle $u$.
Also,
Lemma \ref{lemmakreuz} applies to $\overline \pi_t(u_\pm )= u_\pm^{(t)}$.
%
Thus by Lemma \ref{lemmakreuz} and diagram (\ref{com1}) we get
\begin{eqnarray*}
F(\pi_t)  \circ \Psi_x &=&
F(\pi_t) \circ {u_-}_\#^{-1} \circ F(j)^{-1} \circ \big (u_\# \circ F(s^+) - F(s^-) \big) \\
&=& { u_-^{(t)}}_\#^{-1}  \circ F(\pi_t) \circ F(j)^{-1} \circ \big (u_\# \circ F(s^+) - F(s^-) \big) \\
&=& {u_-^{(t)} }_\#^{-1}  \circ F(j^{(t)})^{-1}  \circ F(\lambda_t) \circ \big (u_\# \circ F(s^+) - F(s^-) \big) \\
&=& { u_-^{(t)}}_\#^{-1}  \circ F(j^{(t)})^{-1} \circ \big ( u^{(t)}_\# 
\circ F(\lambda_t)  \circ F(s^+) -
F(\lambda_t)  \circ F(s^-) \big)   \\
&=& { u_-^{(t)}}_\#^{-1}  \circ F(j^{(t)})^{-1} \circ \big ( u^{(t)}_\#  \circ F({s^+}^{(t)}) -
F({s^-}^{(t)}) \big)
\; = \;\Psi_{x_t}.
\end{eqnarray*}
By homotopy invariance of $F$ we have $F(\pi_0)= F(\pi_1)$ and thus
$\Psi_{x_0} = \Psi_{x_1}$.
\end{proof}

\begin{lemma}    \label{lemmadegenerate}
Let $x,d \in \E^G(A,B)$ where $d$ is degenerate.
Then $\Psi_{x+d} = \Psi_x$.
\end{lemma}

\begin{proof}
Like the proof of Lemma \ref{lemmahomotopy} the proof is rather insensitive
between the group and inverse semigroup case, and it is also similar to the
proof of Lemma \ref{lemmahomotopy}, so we omit the details
and refer to Thomsen's paper \cite{thomsen}.
\end{proof}

We may summarize Lemmas \ref{lemmahomotopy} and \ref{lemmadegenerate}
as follows.

\begin{corollary}
The map $\Psi$ canonically induces a map
$$\Psi:\F^G(A,B) \rightarrow \mbox{Hom}(F(A),F(B))$$
by $\Psi_{[x]} := \Psi_x$ for $x \in \E^G(A,B)$.

\end{corollary}

\begin{lemma}    \label{lemmafunctoriality}
For every unital $*$-homomorphism 
$\lambda: (C,\gamma) \rightarrow (D, \delta)$ (where $C$ and $D$ are 
unital)
one has
$$F(\lambda \otimes \mbox{id}_\calk) \circ \Psi_{[x]} = \Psi_{(\lambda \otimes \mbox{id}_\calk)_*[x]} $$
for $[x] \in \E^G \big((A,\alpha), (C \otimes \calk, \gamma \otimes \mbox{triv}) \big)$.
\end{lemma}

\begin{proof}
The proof is similar to the proof of Lemma \ref{lemmahomotopy}, and
rather insensitive between the group and inverse semigroup case,
and thus we refer to Thomsen's paper \cite{thomsen}.
One uses Lemma \ref{lemmaphifunct}.
\end{proof}

\section{The abelian group homomorphism $\Psi'$}

\label{sectionPsiprime}

In this section we shall define a variation $\Psi'$ of the map $\Psi$
so as that the functoriality of Lemma \ref{lemmafunctoriality}
holds also in the non-unital case.
It will then follow that $\Psi'$ is an abelian group homomorphism.
We shall also remove the stability restriction on $B$.
Also $\Phi^{-1}$ will be implemented in order to switch from $\E^G$ to $KK^G$.

From now on $B$ need not longer be stable!

\begin{definition}
{\rm
Fix a one-dimensional projection $e$ in $(\calk,\mbox{triv})$ and define $c_A:A \rightarrow A \otimes \calk$
to be the corner embedding $c_A(a)= a \otimes e$ for all objects $A$ in ${\bf C^*}$.
}
\end{definition}

\begin{definition}    \label{psiprime}
{\rm
Consider the canonical split exact sequence
\begin{displaymath}    	
0 \longrightarrow \big (B \otimes \calk,\beta \otimes \mbox{triv} \big ) \stackrel{j_B}{\longrightarrow} 
\big (B^+ \otimes \calk,\beta^+  \otimes \mbox{triv} \big )
\stackrel{p_B}{\longrightarrow}  \big (C^*(E) \otimes \calk,\tau  \otimes \mbox{triv} \big ) \longrightarrow 0,
\end{displaymath}
where $(B^+,\beta^+)$ denotes the $G$-equivariant unitization
of $(B,\beta)$, see Definition \ref{defunitization}.
Let $F$ be a stable, homotopy invariant, split-exact functor from ${\bf C^*}$
to ${\bf Ab}$.
For every $z \in KK^G(A,B)$ we define an abelian group homomorphism
$$\Psi'_{z} : F(A,\alpha) \rightarrow F(B,\beta)$$
by
$$\Psi'_{z} =  F(c_B)^{-1} \circ F(j_B)^{-1}  \circ \Psi_{ {j_B}_* {c_B}_* \Phi^{-1}(z) }.$$
}
\end{definition}

The occurrence of $F(j_B)^{-1}$ is here valid,
as $F(p_B) \circ \Psi_{ {j_B}_* {c_B}_* \Phi^{-1}(z)}
= \Psi_{{p_B}_* {j_B}_* {c_B}_* \Phi^{-1}(z)} = \Psi_{[0]} = 0$
by Lemma \ref{lemmafunctoriality} and $F(j_B)$ is injective by split-exactness of $F$.

\begin{lemma}     \label{lemmapsi2funct}
For any 
$*$-homomorphism $\lambda:(B,\beta) \rightarrow (C,\gamma)$
one has
$$\Psi'_{\lambda_* (z)}  = F(\lambda) \circ \Psi'_{z}.$$
\end{lemma}

\begin{proof}
By Definition \ref{deffunctorialityE}, $\lambda_*$ commutes with $\Phi^{-1}$, and
by Lemma \ref{lemmafunctoriality} we get
\begin{eqnarray*}
\Psi'_{\lambda_* (z)} &=& F(c_C)^{-1} \circ F(j_C)^{-1}  \circ \Psi_{  {j_C}_* {c_C}_*  \lambda_* \Phi^{-1}(z) } \\
&=& F(c_C)^{-1} \circ F(j_C)^{-1}  \circ \Psi_{ (\lambda^+ \otimes id_\calk)_* {j_B}_* {c_B}_*  \Phi^{-1}(z) } 
= F(\lambda) \circ \Psi'_{z} .
\end{eqnarray*}
\end{proof}

\begin{lemma}    \label{lemmapsi2group}
The map
$$\Psi':KK^G(A,B) \rightarrow \mbox{Hom}(F(A,\alpha), F(B,\beta))$$
is an abelian group homomorphism.
\end{lemma}

\begin{proof}
If $\Psi'$ was additive then by the functoriality of
Lemma \ref{lemmapsi2funct} the collection of the maps $\Psi'$ would form a natural transformation
from the functor $KK^G(A,-): {\bf C^*} \rightarrow {\bf Ab}$ to the functor
$\mbox{Hom}(F(A),F(-)): {\bf C^*} \rightarrow {\bf Ab}$.
Both functors are stable, homotopy invariant and split-exact by  
Corollary \ref{corollaryProp}.
Then \cite[Lemma 3.2]{higson} states in the non-equivariant case that in such a situation the map
$\Psi'$ is automatically additive. The general proof works verbatim also $G$-equivariantly
in our setting.
\end{proof}

\begin{lemma}    \label{lemma1a}
Let $(A,\alpha)$ be a $G$-algebra. 
Then
$$\Psi'_{1_A} = \mbox{id}_{F(A,\alpha)}.$$

\end{lemma}

\begin{proof}
By Definition \ref{deffunctorialityE} one has
$${j_A}_* {c_A}_* \Phi^{-1}(1_A) = \Phi^{-1}({j_A}_* {c_A}_* 1_A)
= \Phi^{-1}([
j_A c_A, A^+ \otimes \calk,0]) = [j_A c_A,0,\nu,\nu]$$ 
in $\E^G(A, A^+ \otimes \calk)$,
where the last identity may be chosen by Definition \ref{defPhi}, and where
%
$\nu:G\rightarrow \calm(A^+ \otimes \calk)$ is the cocycle $\nu_g := \alpha^+_{g g^{-1}} \otimes \mbox{id}$.

Consider now Definition \ref{defpsi} with respect to $(\varphi_\pm,u_\pm):= (j_A c_A,0,\nu,\nu)$.
We have $\Gamma^+ = \Gamma^- = \alpha^+_{g g^{-1}} \otimes \mbox{id}$, $u = u_+ = u_- = \nu$
and ${u_-}_\#=\mbox{id}$, $u_\#=\mbox{id}$.  
%
Hence
\begin{eqnarray*}
\Psi'_{1_A} &=& F(c_A)^{-1}  \circ F(j_A)^{-1} \circ \Psi_{[j_A c_A,0,\nu,\nu]}  \\
&=&
F(c_A)^{-1}  \circ F(j_A)^{-1} \circ {u_-}_\#^{-1} \circ F(j)^{-1} \circ \big (u_\# \circ F(s^+) - F(s^-) \big) \\
&=&
F(c_A)^{-1}  \circ F(j_A)^{-1}  \circ F(j)^{-1}
\circ \big (F(s^+) - F(s^-) \big) \; = \; \mbox{id}_{F(A)}
\end{eqnarray*}
as the difference $s^+ - s^- = j \circ j_A \circ c_A$ happens to be a $*$-homomorphism and
thus $F(s^+) - F(s^-) = F(s^+ - s^-)$.
\end{proof}

\section{The natural transformation $\xi$}

\label{sectionxi}

In this section we shall show 
Theorem \ref{theoremThomsen}.



\begin{definition}    \label{defxi}
{\rm
Let $F$ be a stable, homotopy invariant, split-exact functor from ${\bf C^*}$
to ${\bf Ab}$.
Let $d \in F(A,\alpha)$. 
There is a natural transformation
$$\xi : KK(A , - ) \rightarrow F(-)$$
defined by
$$\xi_B(z) =  \Psi'_z (d)$$ 
for $z \in KK^G(A,B)$.
}
\end{definition}

That $\xi$ is a natural transformation follows from 
Definition \ref{deffunctorialityE}, and Lemmas
\ref{lemmapsi2funct} and \ref{lemmapsi2group}.



\begin{lemma}    \label{lemmaK1a}
Consider the maps $\Psi$ and   
$$\Psi':KK^G(A,B) \rightarrow \mbox{Hom}\big(KK^G(A,A),KK^G(A,B) \big)$$
developed in Definitions \ref{defpsi} and \ref{psiprime}, 
respectively,
for the homotopy invariant, stable, split-exact functor $F(-)=KK^G(A,-)$ from ${\bf C^*}$ to ${\bf Ab}$.
Then
$$\Psi'_{z} (1_A) = z$$
for all $z \in KK^G(A,B)$.
\end{lemma}

\begin{proof}
Let $x=(\varphi_\pm,u_\pm) \in \E^G(A,B)$, where $B$ is stable.
Then we compute
\begin{eqnarray*}
&& \Psi_{x} (1_A)  \\ 
&=& {u_-}_\#^{-1} \circ F(j)^{-1} \circ \big (u_\# \circ F(s^+) - F(s^-) \big) (1_A) \\
&=&
{u_-}_\#^{-1} \circ F(j)^{-1}  \big (u_\#  [(id_A,\varphi_+) ,(A_x,\Gamma^+),0]
-
[(id_A,\varphi_-) ,(A_x,\Gamma^-),0]  \big )\\
&=&
{u_-}_\#^{-1} \circ F(j)^{-1}  \big ([(id_A,\varphi_+) ,(A_x,\Gamma^+ u^*),0]
- [(id_A,\varphi_-) ,(A_x,\Gamma^-),0]  \big)  \\
&=&
{u_-}_\#^{-1} \big ((\varphi_+,\varphi_-), (B \oplus B, (u_+ \beta u_-^*, u_- \beta u_-^*) ), T\big )  \\
&=&
\big ((\varphi_+,\varphi_-), (B \oplus B, (u_+ \beta, u_- \beta) ),  T \big )  \\
&=& \Delta(x) = \Phi([x]),
\end{eqnarray*}
where $B \oplus B$ is equipped with the grading $\epsilon(x,y)=(x,-y)$
and $T$ is the flip operator.
Then, with Definitions \ref{psiprime} and \ref{deffunctorialityE},
\begin{eqnarray*}
\Psi'_{z} (1_A) &=& F(c_B)^{-1} \circ  F(j_B)^{-1} \circ \Psi_{{j_B}_* {c_B}_* \Phi^{-1}(z)} (1_A)\\
&=& F(c_B)^{-1} \circ  F(j_B)^{-1} \circ \Phi \big ( {j_B}_* {c_B}_* \Phi^{-1}(z) \big )  = z .
\end{eqnarray*}
\end{proof}

\begin{proposition}    \label{propuniqueness}
Given $F$ and $d$ as in Definition \ref{defxi}, $\xi$ is the only
existing natural transformation from $KK^G(A,-)$ to $F(-)$ such that
$$\xi_A(1_A) =d.$$ 
\end{proposition}

\begin{proof}
That $\xi_A(1_A) =d$ follows from Lemma \ref{lemma1a}.
It remains to prove uniqueness of $\xi$.


Now consider another natural transformation $\eta:KK(A,-) \rightarrow F(-)$
such that $\eta_A(1_A) = d$.
Define $K(-)=KK^G(A,-)$.  
Denote the $\Psi$ for $K$ by $\Psi^{(K)}$ for clearity.

We have a commuting diagram
$$
\xymatrix{
KK^G(A,C) \ar[r]^{K(f)} \ar[d]_{\eta_C} &  KK^G(A,D)   \ar[d]^{\eta_D}  \\
F(C) \ar[r]^{F(f)} &  F(D)  \\
}
$$
for all homomorphisms $f \in {\bf C^*}(C,D)$.
Since $\Psi^{(K)}_{[x]}$ and ${\Psi^{(K)}_{z}}'$ 
are only compositions of such maps
$K(f)$, we also have
\begin{equation}     \label{etapsi}
\eta_B \circ {\Psi^{(K)}_z}'  = \Psi'_z \circ \eta_A .
\end{equation}


Thus
$$\eta_B  ( z  ) = \eta_B \big ({\Psi^{(K)}_{z}}'(1_A) \big)
= \Psi'_{z} \big(\eta_A(1_A) \big ) = \Psi'_{z}(d) = \xi_B  (z  )$$
%
by Lemma \ref{lemmaK1a}.
\end{proof}

Definition \ref{defxi} and Proposition \ref{propuniqueness}
sum then up to:

\begin{corollary}
Theorem \ref{theoremThomsen} is true.
\end{corollary}

\section{The universality theorem}

\label{sectionuniversal}

In this section we shall deduce Theorem \ref{theoremHigson}
as described in \cite[Theorem 4.5]{higson}.

\begin{definition}    \label{defFadditive}
{\rm
A functor $F: {\bf C^*} \rightarrow {\bf A}$ into an additive category
is called split-exact, homotopy invariant and stable if the
functor $H^A(-)=\mbox{Hom}(F(A),F(-))$ from ${\bf C^*}$ to the abelian groups
has these properties for all objects $A$ in ${\bf C^*}$.
}
\end{definition}

For convenience of the reader we recall 
another characterization
of split-exact, homotopy invariant, stable functors into additive categories,
see \cite[p. 269]{higson}.

\begin{lemma}   \label{lemmaFadditivecat}
A functor $F: {\bf C^*} \rightarrow {\bf A}$ into an additive category
${\bf A}$ is stable, homotopy invariant and split-exact if and only if
\begin{itemize}
\item[(a)]
$F(f)$ is invertible for every 
corner embedding
$f \in {\bf C^*} (A,  A \otimes \calk)$,

\item[(b)]
$F(f)=F(g)$ for all homotopic $f,g \in {\bf C^*}(A,B)$, and

\item[(c)]
for every split exact sequence
$$\xymatrix{
0 \ar[r] & A \ar[r]^{j} &
D \ar@<.5ex>[r]^{p}
&
B \ar[r] \ar@<.5ex>[l]^{s} & 0
}
$$
the map
$F(A) \oplus F(B) \rightarrow F(D)$ defined by
$$F(j) \circ p_1 + F(s) \circ p_2$$
is
an isomorphism, where $p_1,p_2$ denotes the projection maps.

\end{itemize}
\end{lemma}

\begin{lemma}   \label{lemmafunctK}
Consider $\Psi'$ for the 
functor $F(-)=KK^G(A,-)$.
Then
$$\Psi'_{w} (z) = z \otimes_B w$$
for all $w \in KK^G(B,C)$ and $z \in KK^G(A,B)$.
%
\end{lemma}

\begin{proof}
By the functoriality of the Kasparov product and Lemma \ref{lemmaK1a} we have
\begin{eqnarray*}
&&  \Psi_{w}' (z) = F({c_C})^{-1} \circ F({j_C})^{-1} \circ \Psi_{ {j_C}_* {c_C}_* \Phi^{-1}(z)}  (z \otimes_B 1_B)  \\ 
&=& F({c_C})^{-1} \circ F({j_C})^{-1}  \circ {u_-}_\#^{-1} \circ F(j)^{-1} \circ \big (u_\# \circ F(s^+) - F(s^-) \big) (z \otimes_B 1_B) \\
&=&  z \otimes_B \Psi_{w}' (1_B)   = z \otimes_B w .
\end{eqnarray*}
Actually, the last $\Psi'$ refers to the functor $F(-)=KK^G(B,-)$.
%
%
\end{proof}

\begin{theorem}
Theorem \ref{theoremHigson} is true.

\end{theorem}

\begin{proof}
Consider a functor $F$
as in Definition \ref{defFadditive}.
Let $A$ be an object in ${\bf C^*}$.
Apply Definition \ref{defxi} to the split-exact, stable, homotopy invariant functor
$H^A: {\bf C^*} \rightarrow {\bf Ab}$ defined by
$$H^A 
(-) = \mbox{Hom}(F(A),F(-))$$
and the element $d = 1_{F(A)} \in H^A(A)$.  

We obtain a natural transformation
\begin{equation}     \label{nt1}
\xi^A:KK^G(A,-) \rightarrow \mbox{Hom}(F(A),F(-)).
\end{equation}

Define the functor $\hat F: {\bf K^G} \rightarrow {\bf A}$ by
\begin{equation}     \label{nt2}
\hat F(z) = \xi_B^A(z)
\end{equation}
for all $z \in KK^G(A,B)$.
By Proposition \ref{propuniqueness},    
$$\hat F(1_A)= \xi_A^A(1_A)= 1_{F(A)}.$$
Since by definition
$$H^A(f)(w) = F(f) \circ w$$
for $f \in {\bf C^*}(B,C)$ and $w \in \mbox{Hom}(F(A),F(B))$,
and $\Psi'_z$ is just a composition of such $H^A(f)$s,
notice that
\begin{equation}   \label{psicirc}
\hat F(z)= \Psi'_{z} (1_{F(A)}) = \Psi'_{z} \circ 
1_{F(A)} = \Psi'_{z}  \quad \in \quad \mbox{Hom}(F(A),
F(B)).
\end{equation}

We compute the functoriality of $\hat F$ as follows.
Consider $z \in KK^G(A,B)$ and  $w \in KK^G(B,C)$.
Then with Lemma \ref{lemmafunctK}, identity (\ref{etapsi}) and (\ref{psicirc}) we 
compute
\begin{eqnarray*}
&& \xi^A_C \big (z \otimes_B w \big ) = \xi_C^A \big ( \Psi'_{w} (z) \big )
= \Psi'_{w} \big (\xi_B^A(z) \big )
= \Psi'_{w} \big (\hat F(z) \big )\\
&=& 
\Psi'_{w} 
\circ \hat F(z)
= \Psi'_{w} (1_{F(A)})  \circ \hat F(z)
= \hat F(w) \circ \hat F(z)  .
\end{eqnarray*}

Let $f \in {\bf C^*} (A,B)$.
By (\ref{psicirc}), 
Definition \ref{lemmaphifunct} and Lemma \ref{lemmapsi2funct} we have
$$\hat F(\kappa(f))= \hat F(f_*(1_A)) = \Psi'_{f_*(1_A)} = F(f) \circ \Psi'_{1_A}
= F(f) \circ \hat F(1_A) = F(f).$$

We are going to show uniqueness of $\hat F$.
Let now $\hat F:{\bf C^*} \rightarrow {\bf A}$ be any given functor with $\hat F \circ \kappa = F$.
For $z \in KK^G(A,B)$ and $f \in {\bf C^*}(B,C)$ we then have
$$\hat F( f_*(z)) = 
\hat F(z \otimes_B f_*(1_B)) = F(f) \circ \hat F(z)  = H^A(f) \big (\hat F(z) \big ).$$
Hence (\ref{nt2}) defines a natural transformation (\ref{nt1}), which by
Proposition \ref{propuniqueness}  
is uniquely determined.
Hence $\hat F$ is uniquely determined.
%
%
\end{proof}

\section{Non-unital inverse semigroups}

\label{sectionnonunital}

\begin{corollary}
Corollary \ref{corollarynonunital} is true.
\end{corollary}

\begin{proof}
If $G$ is declared to be a non-unital
inverse semigroup
(even it may have a unit), then we define $G$-algebras and
$KK^G$-theory as before, with the only difference that the $G$-action
$\alpha:G \rightarrow \mbox{End}(A)$ on a $C^*$-algebra is not required to be
unital, 
and similar so for Hilbert modules.
Then we adjoin unconditionally a unit $1$ to $G$ to obtain $G^+:= G \sqcup \{1\}$
and regard it as a unital inverse semigroup.
Then every non-unital $G$-action can be extended to a unital $G^+$-action,
and every unital $G^+$-action can be restricted to a non-unital $G$-action.
This one-to-one correspondence shows that the $C^*$-categories
${\bf C^*_G}$ and ${\bf C^*_{G^+}}$ and the $KK$-theories
${\bf K^G}$ and ${\bf K^{G^+}}$ are the same
in a trivial way.
Corollary \ref{corollarynonunital} follows then by applying Proposition \ref{propositionFunctor}
and Theorems \ref{theoremThomsen} and \ref{theoremHigson} to $G^+$.

Similarly we may view a countable discrete groupoid $\calg$ as an
inverse semigroup $G:= \calg \sqcup \{0\}$ by adjoing a zero element $0$,
which always has to act as the zero operator on $C^*$-algebras and Hilbert modules
as already noted,
and where $g h:=0$ in $G$ if $g,h \in \calg$ are incomposable.
\end{proof}

{\bf Acknowledgement.}
We thank the 
the Universidade Federal de Santa Catarina in Florian\'opolis
for the support we received when developing the content
of this paper in 2013. 
We have published a short version 
in arXiv under the same title 
in 2014.
Since 
that version turned out to be 
too brief and sketchy for an official
publication,
we decided to write this self-contained paper.
We thank Alain Valette for suggesting to write a 
self-contained version.

\bibliographystyle{plain}
\bibliography{references}

\end{document}